\documentclass[12pt,reqno]{amsart}

\usepackage{amsmath,amsthm}
\usepackage{amssymb}
\usepackage{enumerate}

\newcommand{\eps}{\varepsilon}
\renewcommand{\gcd}{\operatorname{gcd}}

\newcommand{\e}{\mathrm{e}}

\newcommand{\dx}{\,\mathrm{d}x}
\newcommand{\dy}{\,\mathrm{d}y}

\newcommand{\ua}{\boldsymbol{\alpha}}
\newcommand{\ulam}{\boldsymbol{\lambda}}

\newcommand{\Aa}{\mathcal{A}}
\newcommand{\Hh}{\mathcal{H}}
\newcommand{\Ss}{\mathcal{S}}
\newcommand{\Rr}{\mathcal{R}}
\newcommand{\Ll}{\mathcal{L}}
\newcommand{\Nn}{\mathcal{N}}
\newcommand{\T}{\mathcal{T}}

\newcommand{\Ma}{\mathfrak{M}}
\newcommand{\ma}{\mathfrak{m}}

\newcommand{\N}{\mathbb{N}}

\newcommand{\Z}{\mathbb{Z}}
\newcommand{\R}{\mathbb{R}}
\newcommand{\C}{\mathbb{C}}

\newcommand{\uu}{\mathbf{u}}
\newcommand{\uv}{\mathbf{v}}

\newcommand{\um}{\mathbf{m}}
\newcommand{\un}{\mathbf{n}}
\newtheorem{theorem}{Theorem}
\newtheorem{lemma}[theorem]{Lemma}

\newtheorem{conjecture}[theorem]{Conjecture}
\theoremstyle{definition}
\newtheorem{definition}[theorem]{Definition}
\allowdisplaybreaks

\begin{document}

\title{Bounds for discrete moments of Weyl sums 
and applications}

\author[K.~Halupczok]{Karin Halupczok}

\address{Karin Halupczok,  Mathematisches Institut der
HHU D\"{u}sseldorf,
Universit\"{a}tsstra\ss{}e 1, D-40225 D\"{u}sseldorf, Germany}

\email{karin.halupczok@uni-duesseldorf.de}
\date{}

\begin{abstract}
We prove two bounds for discrete moments of Weyl sums.
The first one can be obtained using a standard approach.
The second one involves an observation how this method
can be improved, which leads to a sharper bound in 
certain ranges.
The proofs both build on the recently proved 
main conjecture for Vinogradov's mean value theorem.

We present two selected applications:
First, we prove a new $k$-th derivative test
for the number of integer points close to a curve
by an exponential sum approach. 
This yields a stronger bound than existing results 
obtained via geometric methods, but it is only applicable 
for specific functions. 
As second application we prove a new improvement
of the polynomial large sieve inequality for one-variable
polynomials of degree $k\geq 4$. 
\end{abstract}

\makeatletter
\@namedef{subjclassname@2010}{%
  \textup{2010} Mathematics Subject Classification}
\makeatother



\maketitle

\markboth{\textsc{Karin Halupczok}}{\textsc{Bounds for discrete moments of Weyl sums}}

\section{Introduction}
\label{sec:intro}

The recent breakthroughs of Bourgain, Demeter and Guth in \cite{BDG}
and Wooley in \cite{Woo,WooNest}
has led to a full proof of the main conjecture in Vinogradov's Mean
Value Theorem (VMVT for short).
As one consequence among many, new estimates for Weyl sums are available.
With a standard approach, in this article, we show
that these already lead to strong estimates for moments of Weyl sums
(see Theorem~\ref{Satz2}).

In this context we record an observation that for moments of Weyl
sums, a small extra-improvement can be made using Montgomery's
so-called alternative derivation from \cite[\S 4]{Mont} incorporating
VMVT (see Theorem~\ref{Satz1}).
This additional gain can be exploited in a certain range
for the approximating denominator (see \eqref{eq:impr_range})
assuming the length of summation is large enough.
We formulate a conjecture stating where this gain might lead to, if 
further refinements were available (see Conjecture~\ref{conj:c}).

Then, in Section~\ref{sec:mvexposums}, the Weyl sum moment estimates
are used to prove $k$-th derivative tests
for discrete moments of exponential sums with smooth functions
(Theorems~\ref{th:Thb} and \ref{th:Thb2}, the mentioned
extra-improvement is incorporated in Theorem~\ref{th:Thb}).

The achieved bounds for moments of Weyl sums and exponential sums with
smooth functions lead to improvements in some number-theoretic
applications, and we present two such applications. 

The first application, discussed in Section~\ref{sec:appl1},
is the problem of counting integer points close to smooth curves. 
For this, we use a new approach involving exponential sums such 
that strong bounds for the counting quantity $\Rr(f,N,\delta)$, see
Definition~\ref{defR}, can be obtained.
The novelty is to perform an efficient Weyl shift step over a 
set of indices known to be $H'$-spaced so that
the cluster structure of the indices is respected. This allows a
saving of an extra factor $H^{-1}$ in the proof of
Theorem~\ref{th:thR}, see the arguments before \eqref{eq:eqsuma}.

Compared to some existing bounds the resulting
bound in Theorem~\ref{th:thR}
is stronger, but is valid only for certain appropriate
functions. This is discussed at the end of the section.

The second application, discussed in Section~\ref{sec:polyLSI}, concerns
the polynomial large sieve inequality from \cite{KH,KH2}.
In the one-dimensional case we obtain a new improvement 
of the bound. That new bound comes from the extra-improvement in
Theorem~\ref{Satz1}. 

\subsection{Notations and conventions}
\label{ssec:notations} 
Let $k$ denote a fixed positive integer and
let $\eps$ be an arbitrary small positive real number that may change
its value during calculations. By $s,s_{0},s_{1}\geq 1$ 
we denote integers that depend on $k$.
In this article,
we suppress the dependence of the implicit constants
on $k$, $s$ or $\eps$ in our notation, simply writing
$\ll$ for $\ll_{k,s,\eps}$.
Moreover, we write $f\lll g$ if $f(x)=o(g(x))$, that is if
$f(x)/g(x)\to 0$ for $x\to \infty$.

For $\alpha\in\R$ we write
$\e(\alpha):=\exp(2\pi i \alpha)$ for
the complex exponential function and $\|\alpha\|$ denotes the distance
from $\alpha$ to the nearest integer.

For integers $k\geq 1$, $s\geq 1$ and a real number $x>0$ we use the
notation $J_{k}(x,s)$ for Vinogradov's integral, that is
the number of solutions to Vinogradov's system
\begin{align*}
  m_{1}+\dots+m_{s}&= n_{1}+\dots+n_{s} \\
m_{1}^{2}+\dots+m_{s}^{2}&= n_{1}^{2}+\dots+n_{s}^{2} \\
&\cdots\\
m_{1}^{k}+\dots+m_{s}^{k}&= n_{1}^{k}+\dots+n_{s}^{k}  
\end{align*}
with $1\leq m_{1},\dots,m_{s},n_{1},\dots,n_{s}\leq x$.
In this work, although no use of the integral representation of
$J_{k}(x,s)$ is made, the $\ell_{2}$-norm of the counting
function $r_{s}(\ulam)$ is used in the proof of 
Theorem \ref{Satz1}.

Given a positive integer $n$ we write $\tau(n)$ for the number of 
divisors of $n$, and $\tau_{3}(n)$ denotes the number of ways one can
write $n$ as a product of $3$ factors. We will use the well-known
estimates $\tau(n)\ll n^{\eps}$ and $\tau_{3}(n)\ll n^{\eps}$.

The set of real functions with continuous derivatives of
order up to $k$ on an interval $I$ is denoted by $C^{k}(I)$.


\subsection{Auxiliaries}
\label{ssec:auxlemmas}

We collect some auxiliary results needed as tools in 
this article.

The following is the well-known sum lemma, 
see e.g.~\cite[Lemma 4C]{Schm} for a proof.

\begin{lemma}[Sum lemma]
\label{lem:sumlemma}
  For $\alpha\in\R$ let $u,q$ be integers with $(u,q)=1$, $0\leq u\leq
  q-1$ and $|\alpha-u/q|<q^{-2}$. Let $\beta\in\R$. Then
\[
   \sum_{Z\leq h\leq Y} \min(X,\|\alpha h +\beta\|^{-1}) \ll
      (X+q\log(q))((Y-Z)q^{-1}+1).
\]
\end{lemma}

For further improvements, we will also make use of the following
result from \cite[Lemma 9C]{Schm}.
\begin{lemma}[Variant of sum lemma]
\label{lem:varsumlemma}
  For $\alpha\in\R$ let $u,q$ be integers with $(u,q)=1$, 
$0\leq u\leq q-1$ and $|\alpha-u/q|<q^{-1}X^{-1}$. Let $\beta\in\R$. Then
\[
   \sum_{1\leq j\leq q} \min(X,\|\alpha j +\beta\|^{-1}) \ll
      \min(X,q\|\beta q\|^{-1}) + q\log q.
\]
\end{lemma}

Next, we need the following simple bound
for the number of curve points close to integer points.
This is Lemma 2 in \cite{HB-V}, see also \cite[Thm. 5.6]{Bor}, where
a proof is provided.
\begin{lemma}[Curve points close to integer points]
\label{lem:smf}
Let $N$ be a positive integer, and suppose that $g(x) : [0, N] \to \R$ has
a continuous derivative on $(0, N )$. Suppose further that
$0 < \lambda \leq g'(x) \leq A \lambda$ for all $x \in (0, N)$.
Then
\[
   \#\{n \leq N : \|g(n)\| \leq \delta\} 
  \ll (1 + A \lambda N )(1 + \delta/\lambda).
\]
\end{lemma}

We also need the following simple assertion.
\begin{lemma}
  \label{lem:sri}
  Consider positive real functions $S,f$ and $A$.
  Assume that $S(x)\ll f(x)S(x) + A(x)$. If $f(x)$ tends to zero for
  $x\to \infty$, then $S(x)\ll A(x)$.
\end{lemma}
\begin{proof}
  If $C$ denotes the implicit constant, then $S(x)(1-Cf(x))\leq C
  A(x)$. With $f(x)\leq C/2$ for all large $x$ we deduce $S(x)\leq 2C A(x)$.
\end{proof}

Another important ingredient is 
Vinogradov's Mean Value Theorem. The theorem is elementary for $k=1$
and $k=2$. For the highly nontrivial cases $k\geq 3$ it has been
proved in \cite{BDG} by Bourgain, Demeter and Guth for $k\geq 4$ 
and in \cite{Woo} by Wooley for $k=3$, and again in \cite{WooNest} by Wooley 
for $k\geq 3$. In our analysis, we will make use of this deep estimate.
\begin{theorem}[VMVT] 
\label{th:VMVT}
Let $s\geq 1$, $k\geq 1$ be integers and
  $\eps>0$. Then
  $J_{k}(x,s)\ll  (x^{s} + x^{2s-k(k+1)/2})x^{\eps}$.
\end{theorem}

\section{Discrete moments of Weyl sums}
\label{sec:WeylVMVT}

It is known that VMVT has the following impact on
Weyl sum estimates:
\begin{theorem}[Weyl sum estimate]
\label{th:weyloriginal}
  Let $P\in\R[X]$ be a polynomial of degree $k\geq 2$, and for
the leading coefficient $\alpha_{k}$ of $P$ let $u,q$ be integers with
 $(u,q)=1$, $q\geq 1$ and $|\alpha_{k}-u/q|<q^{-2}$.
Then we have
  \[
    \sum_{n\leq x} \e(P(n)) \ll  x \Big(\frac{1}{q}+\frac{1}{x}
    +\frac{q}{x^{k}}\Big)^{1/k(k-1)} x^{\eps}.
\]
\end{theorem}

A proof can easily be found using Montgomery's
exposition \cite[\S 4]{Mont}. Our analysis of this proof yields a 
generalization of this estimate for discrete moments of Weyl sums. 
By changing a small aspect, it comes  
with an extra-improvement, stated below as Theorem~\ref{Satz1}. 
Compared to this, Theorem
\ref{Satz2} below is just a straight forward generalization of Theorem
\ref{th:weyloriginal} that stems from Montgomery's original approach
presented in \cite[\S 4]{Mont}.

\bigskip
We give the definition of the discrete moments we look at.

\begin{definition}
\label{def1}
Let $k\geq 2$ be a fixed integer and
consider a fixed polynomial $P_{\ua}\in\R[X]$ of degree $k$ with
$P_{\ua}(0)=0$, say
\[
   P_{\ua}(X)=\alpha_{k}X^{k}+\alpha_{k-1}X^{k-1}+\dots+\alpha_{1}X
\]
with $\alpha_{1},\dots,\alpha_{k}\in \R$.
Let $x>1$ be a sufficiently large real number
and for $a\in\N$ let
\[
    S_{a}(\ua):=\sum_{m\leq x} \e(aP_{\ua}(m))
\]
be a corresponding Weyl sum of $P_{\ua}$. The twist with $a$ allows us
to consider discrete moments of the form 
\begin{equation}
\label{eq:DMWeylsum}
\sum_{a\leq T}|S_{az}(\ua)|^{2s}
\end{equation}
with large real $T>1$ and with 
fixed numbers $z,s\in\N$. 
\end{definition}
The role of $z$ is to control a possible dependence of a further
factor in the argument of the exponential. We might think of a small
$z$, or even $z=1$.

Sums of the shape \eqref{eq:DMWeylsum} occur in numerous applications,
like in Dirichlet's divisor problem, counting integer points close to curves, 
or, as we will see below, in the polynomial large sieve
inequality (for one variable polynomials). We will restrict on
presenting just the latter two applications which work well.

Our first goal is to give good estimates for the expression in
\eqref{eq:DMWeylsum} depending on $x$ and $T$.

Note that bounds for other moments can then easily be derived by H\"older's
inequality
\begin{equation}
\label{eq:Hoelder}
   \sum_{a\leq T}|S_{az}(\ua)|^{\ell}\leq T^{1-\ell/2s}
     \Big(\sum_{a\leq T} |S_{az}(\ua)|^{2s}\Big)^{\ell/2s},\ 0<\ell< 2s.
\end{equation}
Since our results use different values of $s$, it is convenient
to state bounds for the first moment, which makes the statements
easy to compare. Therefore, the results Theorem \ref{Satz1} and 
Theorem \ref{Satz2} below are stated for the first moment.

\section{Improved moment estimate}
\label{sec:WeylImpr}

We are following the estimate of Weyl sums along the
lines of Montgomery's so-called alternative derivation in \cite[\S 4.4]{Mont} 
and carry it over to the situation of discrete moments.
This approach yields 
the following result for Weyl sums as given in Definition \ref{def1}.
We call it the improved moment estimate. 
The direct approach leading to Theorem\ \ref{Satz2} 
yields a bound that is weaker in certain ranges. This is discussed in
Subsection \ref{ssec:compconj}.

\begin{theorem}[Improved moment estimate]
\label{Satz1}
Let $k\geq 3$, $s_{0}=(k-1)(k-2)/2+1$ and $u,q$ be integers with $q\geq 1$,
$(u,q)=1$ and $|\alpha_{k}-u/q|<q^{-2}$. Then
\begin{equation*}
  \sum_{a\leq T} |S_{az}(\ua)| \ll Tx
  \Big(\frac{zx^{k-1}}{q}+\frac{zx^{k-1}\log(q)}{T}
+\frac{1}{x}+\frac{q\log(q)}{Tx}\Big)^{1/2s_{0}}x^{\eps}.
\end{equation*}
\end{theorem}

In the bound, we ordered the factors on the right hand side:
it starts with
the trivial estimate $Tx$, then we give the improvement factor and
then a small additional factor $x^{\eps}$.

\begin{proof}
We need to introduce some of the notations from  \cite[\S 4.4]{Mont},
but writing $x$ instead of $N$.

Thus for $j\in\N$, the $j$-th power sum of a tuple
$\um=(m_{1},\dots,m_{s})\in\N^{s}$ is written as
$s_{j}(\um):=m_{1}^{j}+\dots+m_{s}^{j}$,
and the difference of two power sums as
$d_{j}=d_{j}(\uu,\uv):=s_{j}(\uu)-s_{j}(\uv)$, where
$u_{i},v_{i}\in\{-x,\dots,x\}$.

Multiplying $S_{az}(\ua)^{s}$ 
out, sorting the summands according
to the value of the power sums with power $j=1,\dots,k-2$, 
and an application of Cauchy--Schwarz's inequality yields
\begin{equation*}
  |S_{az}(\ua)|^{2s} \leq s^{k-2} x^{(k-1)(k-2)/2} \T(a)
\end{equation*}
with
\begin{align*}
\T(a)&:=\sum_{\substack{\um,\un
      \\ s_{j}(\um)=s_{j}(\un),\ j=1,\dots,k-2}} 
   \e(azP(m_{1}))
   \cdots\e(azP(m_{s}))\\ &\hspace{5cm}\cdot 
    \e(-azP(n_{1}))\cdots\e(-azP(n_{s})) \\
   &=\sum_{\substack{\um,\un
      \\ s_{j}(\um)=s_{j}(\un),\ j=1,\dots,k-2}} 
   \e\big( (s_{k}(\um)-s_{k}(\un)) az\alpha_{k} \\ &\hspace{5cm}+ 
    (s_{k-1}(\um) - s_{k-1}(\un))az\alpha_{k-1}\big) \\
   &= \sum_{\substack{\uu,\uv
      \\ d_{j}(\uu,\uv)=0,\ j=1,\dots,k-2}} 
    \e\big( d_{k} az\alpha_{k} + d_{k-1}az\alpha_{k-1}\big)
    \sum_{m\in I} \e\big(kd_{k-1}maz\alpha_{k}\big),
\end{align*}
compare \cite[Eq.\ (34)--(38) in \S 4.4]{Mont}.
Here, $m$ runs through an interval
$I=I(\uu,\uv,x)$ that contains at most $x$ many successive integers,
and we have put $m=m_{1}$, $m_{i}=m+u_{i}$ for $2\leq i\leq s$,
and $n_{i}=m+v_{i}$ for $1\leq i\leq s$. Note that the vector $\uu$
consists of one variable less, it has $s-1$ components.
In this step, all variables $m_{i}, n_{i}$ in the Vinogradov System
$s_{j}(\um)=s_{j}(\un)$, $j=1,\dots,k-2$, have been translated by
$m=m_{1}$, thus we make use of the translation invariance of the
Vinogradov system.

Now let $h=d_{k-1}\leq 2skx^{k-1}$ and sort the tuples $\uu,\uv$ 
by their value for $d_{k-1}=h$.

Then the summation of $\T(a)$ over $a\leq T$ yields
\begin{multline}
\label{eq:TSumme}
  \sum_{a\leq T}\T(a)  =  \sum_{h\ll x^{k-1}}
  \sum_{\substack{\uu,\uv\\d_{k-1}=h\\d_{j}=0,\ j=1,\dots,k-2}}
     \sum_{m\in I} \sum_{a\leq T}
  \e(az(\alpha_{k}d_{k}+\alpha_{k}khm+h\alpha_{k-1})),
\end{multline}
where the last geometric sum can be estimated by
\[
   \ll \min(T,\|\alpha_{k}zd_{k}+\alpha_{k}zkhm+zh\alpha_{k-1}\|^{-1}).
\]

In the following, the notation $\sum'$ at the sum over $\uu,\uv$ 
abbreviates the condition that $d_{j}=0$ holds for $j=1,\dots,k-2$.

Using this, we obtain
\begin{align*}
  \sum_{a\leq T} \T(a) &\ll \sum_{h\ll x^{k-1}} 
  \sideset{}{'}\sum_{\substack{\uu,\uv \\ d_{k-1}=h}}
  \sum_{m\in I} \min(T, \| \alpha_{k}zd_{k}+\alpha_{k}zkhm
     +zh\alpha_{k-1}\|^{-1}) \\
&\ll \sum_{h\ll x^{k-1}}   \sum_{m\in I'}
  \sideset{}{'}\sum_{\substack{\uu,\uv \\ d_{k-1}=h}}
 \min(T, \| \alpha_{k}zd_{k}+\alpha_{k}zkhm
     +zh\alpha_{k-1}\|^{-1}),
\end{align*}
where we extended the interval $I$ of length at most $x$ to an
interval $I'$ of length at most $3x$, in order to remove the
dependence on the variables $\uu,\uv$ except on $h$, which makes the
separation of summation possible.
We continue with
\begin{multline*}
  \sum_{a\leq T} \T(a) 
\ll
\sum_{d\ll x^{k}} \sum_{h\ll x^{k-1}}  \sum_{m\in I'}
\min(T, \|\alpha_{k}zd+\alpha_{k}zkhm+zh\alpha_{k-1}\|^{-1}) \\ 
\cdot \sideset{}{'}\sum_{\substack{\uu,\uv \\ d_{k-1}=h,\ d_{k}=d}} 1, 
\end{multline*}
and changing $hm$ to $w$ yields
\begin{align}
 \sum_{a\leq T} \T(a) 
&\ll\sum_{d\ll x^{k}} \sum_{w\ll x^{k}}
\sum_{\substack{h\mid w\\ h\ll x^{k-1}}}
\min(T, \| \alpha_{k}z(d+kw)+zh\alpha_{k-1}\|^{-1})\notag\\ &\hspace{7cm}
  \cdot  \sideset{}{'}\sum_{\substack{\uu,\uv \\ d_{k-1}=h,\ d_{k}=d}} 1\notag \\
&\ll\sum_{j\ll zx^{k}} \sum_{\substack{d\ll x^{k} \\ dz\equiv j (zk)}}
\sum_{\substack{h\ll x^{k-1}\\ h\mid (j-zd)/zk}} 
\min(T, \| \alpha_{k}j+zh\alpha_{k-1}\|^{-1})
 \sideset{}{'}\sum_{\substack{\uu,\uv \\ d_{k-1}=h,\ d_{k}=d}} 1 \notag\\
&\ll \sum_{d\ll x^{k}} \sum_{h\ll x^{k-1}} \Big( \sum_{j\ll zx^{k}}  
\min(T, \| \alpha_{k}j+zh\alpha_{k-1} \|^{-1}) \Big)
\sideset{}{'}\sum_{\substack{\uu,\uv \\ d_{k-1}=h,\ d_{k}=d}} 1, \label{eq:cest}
\end{align}
and an application of Lemma \ref{lem:sumlemma} to the sum in large
brackets yields
\begin{align*}
\sum_{a\leq T} \T(a)
&\ll \sum_{d\ll x^{k}} \sum_{h\ll x^{k-1}} (T+q\log(q))(zx^{k}/q+1)
\sideset{}{'}\sum_{\substack{\uu,\uv \\ d_{k-1}=h,\ d_{k}=d}} 1 \\
&= (T+q\log(q))(zx^{k}/q+1)\sideset{}{'}\sum_{\uu,\uv} 1,
\end{align*}
assuming that the integer $q\geq 1$ is such that there exists an integer $u$ 
with $(u,q)=1$ and $|\alpha_{k}-u/q|<q^{-2}$.

Now we shall give an estimate for the last sum.
For $\ulam\in\Z^{k-2}$ let 
\[
 r_{s-1}(\ulam):=\#\{\uu;\ u_{2}+\dots+u_{s}=\lambda_{1},\dots,\ 
     u_{2}^{k-2}+\dots+u_{s}^{k-2}=\lambda_{k-2}\}
\]
and similarly
\[
 r_{s}(\ulam):=\#\{\uv;\ v_{1}+\dots+v_{s}=\lambda_{1},\dots,\ 
     v_{1}^{k-2}+\dots+v_{s}^{k-2}=\lambda_{k-2}\}.
\]
Then the Cauchy--Schwarz inequality yields
\begin{align*}
  \sideset{}{'}\sum_{\uu,\uv} 1 &= \sum_{\substack{\uu,\uv \\
      d_{j}=0,\ j=1,\dots,k-2}} 1 =
  \sum_{\ulam\in\Z^{k-2}}r_{s-1}(\ulam)r_{s}(\ulam) \\
&\leq \Big(\sum_{\ulam}r_{s-1}(\ulam)^{2}\Big)^{1/2}
   \Big(\sum_{\ulam}r_{s}(\ulam)^{2}\Big)^{1/2}=(J_{k-2}(x,s-1)J_{k-2}(x,s))^{1/2},
\end{align*}
and so we obtain
\begin{align*}
   \sum_{a\leq T} \T(a)
&\ll (J_{k-2}(x,s-1)J_{k-2}(x,s))^{1/2}
\Big(\frac{Tzx^{k}}{q}+zx^{k}\log(q)+T+q\log(q)\Big) \\
&=Tx^{k}(J_{k-2}(x,s-1)J_{k-2}(x,s))^{1/2}
\Big(\frac{z}{q}+\frac{z\log(q)}{T}
+\frac{1}{x^{k}}+\frac{q\log(q)}{Tx^{k}}\Big).
\end{align*}
For the desired moment of Weyl sums this yields
\begin{align*}
  \sum_{a\leq T} &|S_{az}(\ua)|^{2s}  \\
&\ll Tx^{(k-1)(k-2)/2+k} 
(J_{k-2}(x,s-1)J_{k-2}(x,s))^{1/2} \\ &\hspace{6cm}
\cdot\Big(\frac{z}{q}+\frac{z\log(q)}{T}
+\frac{1}{x^{k}}+\frac{q\log(q)}{Tx^{k}}\Big) \\
&= Tx^{2s} 
\Big(\frac{J_{k-2}(x,s-1)J_{k-2}(x,s)}{x^{4s-(k-1)(k-2)-2k}}
\Big)^{1/2}\\ &\hspace{6cm}
\cdot \Big(\frac{z}{q}+\frac{z\log(q)}{T}
+\frac{1}{x^{k}}+\frac{q\log(q)}{Tx^{k}}\Big).
\end{align*}

Now that we have VMVT, Theorem \ref{th:VMVT},
at hand, we can apply the best possible
bound for the term in big brackets that includes the Vinogradov integrals.
Choosing $s=s_{0}$ with $s_{0}=(k-1)(k-2)/2+1$, we have
\begin{multline*}
  \frac{J_{k-2}(x,s-1)J_{k-2}(x,s)}{x^{4s-(k-1)(k-2)-2k}} \\ \ll
  x^{s-1}x^{2s-(k-1)(k-2)/2}x^{-4s+(k-1)(k-2)+2k}x^{\eps}=x^{2k-2+\eps}
\end{multline*}
for this value of $s$. (Note that we have $\ll x^{2k-1+\eps}$
when choosing $(k-1)(k-2)/2$ for $s$ instead, 
so the choice $s=s_{0}$ is optimal.)
We arrive at the following estimate.

\begin{equation}
\label{eq:S1}
  \sum_{a\leq T} |S_{az}(\ua)|^{2s_{0}} \ll Tx^{2s_{0}} x^{k-1}
\Big(\frac{z}{q}+\frac{z\log(q)}{T}
+\frac{1}{x^{k}}+\frac{q\log(q)}{Tx^{k}}\Big) x^{\eps}.
\end{equation}

\bigskip
Using H\"older's inequality \eqref{eq:Hoelder}, we obtain an estimate
for the first moment. In this way, we obtain 
the asserted bound from equation \eqref{eq:S1}.
\end{proof}

\bigskip
Theorem \ref{Satz1} has to be compared with the result obtained by the
straight-forward approach, that is, the following bound.

\begin{theorem}[Standard approach estimate]
\label{Satz2}
Let $k\geq 2$, $s_{1}=k(k-1)/2$ and $u,q$ 
be integers with $q\geq 1$, $(u,q)=1$ and 
$|\alpha_{k}-u/q|<q^{-2}$. Then
\begin{equation*}
  \sum_{a\leq T} |S_{az}(\ua)| \ll 
Tx \Big(\frac{z}{q}+\frac{z}{x}+\frac{q}{Tx^{k}}\Big)^{1/2s_{1}}(Txz)^{\eps}.
\end{equation*}
\end{theorem}

Note that with $T=1$ and $z=1$, we get back Theorem 
\ref{th:weyloriginal} above as a special case.

\begin{proof}
Montgomery's original approach in \cite[\S 4.4,p.81,l.15]{Mont} yields
\begin{multline*}
 \sum_{a\leq T} |S_{az}(\ua)|^{2s} \\ \ll x^{(k-1)(k-2)/2}
x^{-1}J_{k-1}(3x,s)\sum_{a\leq T} \sum_{0\leq h\leq 2sx^{k-1}}
\min(x,\| akhz\alpha_{k}\|^{-1}),
\end{multline*}
where the last double sum can be estimated using Lemma \ref{lem:sumlemma}.
Together with the substitution $w=akhz$ this yields
\[
   \ll \sum_{w\leq 2sTzkx^{k-1}} \tau_{3}(w)\min(x,\|w\alpha_{k}\|^{-1})
   \ll \Big(\frac{zTx^{k}}{q} + zTx^{k-1}+q\Big)(Txz)^{\eps},
\]
where there exist integers $u,q$ with $q\geq 1$, $(u,q)=1$ and
$|\alpha_{k}-u/q|<q^{-2}$. We proceed with
\[
 \sum_{a\leq T} |S_{az}(\ua)|^{2s} \ll x^{2s}T
\Big(\frac{J_{k-1}(3x,s)}{x^{2s-k(k-1)/2}}\Big)
\Big(\frac{z}{q}+\frac{z}{x}+\frac{q}{Tx^{k}}\Big)(Txz)^{\eps}.
\]
Next, using VMVT (Theorem \ref{th:VMVT}) with the optimal 
$s=s_{1}=k(k-1)/2$ leads to the estimate
\begin{equation}
\label{eq:S2}
  \sum_{a\leq T} |S_{az}(\ua)|^{2s_{1}} \ll 
  x^{2s_{1}}T\Big(\frac{z}{q}+\frac{z}{x}+\frac{q}{Tx^{k}}\Big)(Txz)^{\eps}.
\end{equation}

Applying H\"older's inequality \eqref{eq:Hoelder} 
to equation \eqref{eq:S2} yields the desired
first moment as given in the assertion of the theorem. \end{proof}

\subsection{Comparison and conjectural considerations}
\label{ssec:compconj}

The expressions in large brackets in Theorems \ref{Satz1} and
\ref{Satz2} show the improvements compared
to the trivial estimate $Tx$. 
They lead to a nontrivial assertion
if $zx^{k-1}\lll q\lll Tx$ in Theorem \ref{Satz1} 
respectively if $z\lll q\lll Tx^{k}$ in Theorem \ref{Satz2}.
Let $s_{0}=(k-1)(k-2)/2+1$ and $s_{1}=k(k-1)/2$.
 
We compare these improvement expressions 
(supposing $z$ is small in this comparison) and obtain
the following assertions.

\begin{enumerate}[1.)]
\item 
In these expressions, we compare the typical dominant
terms,  $x^{-1/2s_{0}}$ (for $zx^{k}\ll q \ll T$)
with $(z/x)^{1/2s_{1}}$ (for $x\ll q\ll zTx^{k-1}$), 
we immediately see that Theorem
\ref{Satz1} yields a sharper estimate in the intersection range 
$zx^{k}\ll q\ll T$.

\item
The dominant term $(zx^{k-1}/q)^{1/2s_{0}}$ for $q\ll \min(zx^{k},T)$
in Theorem \ref{Satz1} is sharper than $(z/x)^{1/2s_{1}}$ for $x\ll q\ll T$,
 if $q\gg z^{\sigma}x^{k-\sigma}$ with 
\begin{equation}
\label{eq:sigma}
\sigma=\sigma_{k}:=1-s_{0}/s_{1}=2/k-2/k(k-1).
\end{equation}
To summarize, with Theorem \ref{Satz1} we obtain an improvement in the range 
$z^{\sigma}x^{k-\sigma}\ll q\ll \min(zx^{k},T)$.

\item
The dominant term $(zx^{k-1}/T)^{1/2s_{0}}$ for $T\ll q\ll zx^{k}$
is sharper than $(z/x)^{1/2s_{1}}$ for $x\ll q\ll zTx^{k-1}$ if $T\gg
z^{\sigma}x^{k-\sigma}$. Thus in the intersection range $T\ll q
\ll zx^{k}$ we obtain an improvement.

\item
The dominant term $(q/Tx)^{1/2s_{0}}$ for $q\gg \max(T,zx^{k})$ 
is sharper than $(z/x)^{1/2s_{1}}$ for $x\ll q\ll zTx^{k-1}$ if $q\ll
Tx^{\sigma}z^{1-\sigma}$. Thus in the intersection range
$\max(T,zx^{k})\ll q\ll Tx^{\sigma}z^{1-\sigma}$, for which $T\gg
z^{\sigma}x^{k-\sigma}$ has to hold necessarily, we obtain an improvement.
\end{enumerate}

To summarize, Theorem \ref{Satz1} yields an improvement 
only if $T\gg z^{\sigma}x^{k-\sigma}$, so
this term $z^{\sigma}x^{k-\sigma}$ turns out to be a critical value
for $T$ from which on we obtain improvements.
Moreover, above conditions on $q$ have to hold, that is the range
\begin{equation}
\label{eq:impr_range}
z^{\sigma}x^{k-\sigma}\ll q\ll Tx^{\sigma}z^{1-\sigma}.
\end{equation}
For any other $q$, Theorem \ref{Satz2} gives a sharper bound.

\bigskip
An observation is that in \eqref{eq:cest} we made a very coarse
estimate. Heuristically, one would expect that it could be doable
with the mean value over $h$. This would provide a gain of an extra
factor $x^{k-1}$ in the estimate. In this way, 
we would save it also in \eqref{eq:S1} and
arrive at the following conjectural bound.

\begin{conjecture}
  \label{conj:c}
For $k\geq 3$, $s_{0}=(k-1)(k-2)/2+1$, and  $|\alpha_{k}-a/q|<q^{-2}$ for
$(a,q)=1$, the estimate
\begin{equation*}
  \sum_{a\leq T} |S_{az}(\ua)| \ll Tx
  \Big(\frac{z}{q}+\frac{z\log(q)}{T}
+\frac{1}{x^{k}}+\frac{q\log(q)}{Tx^{k}}\Big)^{1/2s_{0}}(Txz)^{\eps}
\end{equation*}
is conjectured to hold true.
\end{conjecture}

Compared to Theorem \ref{Satz1}, this would lead
to an improvement factor
$x^{-k/2s_{0}}$ (around $x^{-1/k}$) instead of $x^{-1/2s_{0}}$, provided that the secondary
terms do not matter. 
It is interesting to see that we can indeed improve further towards 
Conjecture~\ref{conj:c} if we assume suitable 
rational approximations to $\alpha_{k}$ and $\alpha_{k-1}$ as follows.

\begin{theorem}[Second improved moment estimate]
\label{Satz3}
Let $k\geq 2$, $s_{2}=k(k-1)/2+1$ and $u,q$ be integers with $q\geq 1$,
$(u,q)=1$ and $|\alpha_{k}-u/q|<q^{-1}T^{-1}$. Further, let $v,w$ be
integers with $1\leq w\leq x^{k-1}z$ and $|zq\alpha_{k-1}-v/w|<w^{-2}$.
Then
\begin{equation*}
  \sum_{a\leq T} |S_{az}(\ua)| \ll Tx
\Big(\frac{x^{k-1}}{qw}  +
\frac{x^{k-1}}{T}+ \frac{1}{xw}+\frac{q}{Tx}\Big)^{1/2s_{2}}x^{\eps}.
\end{equation*}
\end{theorem}

\begin{proof}
We start as in the proof of Theorem~\ref{Satz1}, but continue \eqref{eq:cest}
with
\begin{align}
\label{eq:eq4neu}
 \sum_{a\leq T} \T(a) 
&\ll \sum_{d\ll x^{k}} \sum_{h\ll x^{k-1}} \Big( \sum_{j\ll zx^{k}}  
\min(T, \| \alpha_{k}j+hz\alpha_{k-1} \|^{-1}) \Big) \notag
\sideset{}{'}\sum_{\substack{\uu,\uv \\ d_{k-1}=h,\ d_{k}=d}} 1 \\
&\ll \Big(\sum_{h\ll x^{k-1}}  \sum_{j\ll zx^{k}}  
\min(T, \| \alpha_{k}j+zh\alpha_{k-1} \|^{-1}) \Big)
\max_{h_{0}}\sideset{}{'} \sum_{\substack{\uu,\uv \\ d_{k-1}=h_{0}}} 1.
\end{align}
To handle the last sum, let
\begin{multline*}
 \tilde{r}_{s}(\ulam,h_{0}):=\#\{\uv;\ v_{1}+\dots+v_{s}=\lambda_{1},\dots,\ 
     v_{1}^{k-2}+\dots+v_{s}^{k-2}=\lambda_{k-2},\\
  v_{1}^{k-2}+\dots+v_{s}^{k-2}  =\lambda_{k-2}+h_{0}\},
\end{multline*}
we obtain
\begin{align*}
  \sideset{}{'} \sum_{\substack{\uu,\uv \\ d_{k-1}=h_{0}}} 1
  &= \sum_{\ulam\in\Z^{k-1}}
  \tilde{r}_{s-1}(\ulam,0)\tilde{r}_{s}(\ulam,h_{0})   \\ &\leq
  \Big(\sum_{\ulam}\tilde{r}_{s-1}(\ulam,0)^{2}\Big)^{1/2}
  \Big(\sum_{\ulam}\tilde{r}_{s}(\ulam,h_{0})^{2}\Big)^{1/2}\\
  &=\Big(J_{k-1}(x,s-1)J_{k-1}(x,s)\Big)^{1/2},
\end{align*}
uniformly in $h_{0}$. Now we turn to the sum over $j,h$ in
\eqref{eq:eq4neu}. 
For each block $B=[1+bq,\dots,q-1+bq]$ of consecutive positive
integers with $b\geq 1$, we have 
\[
\sum_{j\in B}\min(T,\|j\alpha_{k}+hz\alpha_{k-1}\|^{-1})
\ll \sum_{1\leq j\leq q} \min(T,\|j\alpha_{k} +hz\alpha_{k-1}\|^{-1}),
\]
since $\min(T,\|j\alpha_{k}+q\alpha_{k}+hz\alpha_{k-1}\|^{-1})\ll
\min(T, \|j\alpha_{k} +hz\alpha_{k-1}\|^{-1})$ holds true 
under the assumption $\|q\alpha_{k}\|<T^{-1}$. 
Therefore  
\begin{align*}
\sum_{h\ll x^{k-1}}  &\sum_{j\ll zx^{k}}  
\min(T, \| \alpha_{k}j+hz\alpha_{k-1} \|^{-1})  \\
&\ll  
(zx^{k}/q+1) \sum_{h\ll x^{k-1}} \sum_{1\leq j\leq q}
\min(T, \|\alpha_{k}j+hz\alpha_{k-1} \|^{-1}) \\
&\ll 
(zx^{k}/q+1) \sum_{h\ll x^{k-1}} (\min(T,q\|hzq\alpha_{k-1}\|^{-1}) +
q\log q),
\end{align*}
where we applied Lemma~\ref{lem:varsumlemma} in the last step.
Writing $m=hz$ and noting that the number of divisors of $m$ is
$\ll x^{\eps}$, we continue using
Lemma~\ref{lem:sumlemma} by
\begin{multline*}
  \ll (zx^{k}/q+1) x^{\eps} q(T/q + w\log w)(x^{k-1}z/w+1) 
 + zx^{2k-1+\eps}+x^{k-1+\eps}q\\
    \ll x^{\eps}z^{2} (x^{k}+q) (T/q+w)x^{k-1}/w  +
    zx^{2k-1+\eps}+x^{k-1+\eps}q \\  \ll 
   x^{\eps}z^{2} Tx^{k}
   \Big(\frac{x^{k-1}}{qw}+\frac{x^{k-1}}{T}+\frac{1}{xw}+\frac{q}{Tx}\Big),
\end{multline*}
supposing $w\leq x^{k-1}z$.
Then together with \eqref{eq:eq4neu}, we arrive at
\begin{multline*}
  \sum_{a\leq T} \T(a)  \ll
  x^{\eps}x^{k}T\Big(\frac{x^{k-1}}{qw}  +
\frac{x^{k-1}}{T}+ \frac{1}{xw}+\frac{q}{Tx}\Big)
  (J_{k-1}(x,s-1)J_{k-1}(x,s))^{1/2}.
\end{multline*}
For the desired moment of Weyl sums this yields now 
\begin{align*}
  \sum_{a\leq T} &|S_{az}(\ua)|^{2s}  \\
&\ll Tx^{(k-1)(k-2)/2+k+\eps} 
(J_{k-1}(x,s-1)J_{k-1}(x,s))^{1/2} \\ &\hspace{1cm}
\cdot \Big(\frac{x^{k-1}}{qw}  +
\frac{x^{k-1}}{T}+ \frac{1}{xw}+\frac{q}{Tx}\Big) \\
&= Tx^{2s+\eps} 
\Big(\frac{J_{k-1}(x,s-1)J_{k-1}(x,s)}{x^{4s-(k-1)(k-2)-2k}}
\Big)^{1/2} 
\cdot \Big(\frac{x^{k-1}}{qw} +
\frac{x^{k-1}}{T}+ \frac{1}{xw} +\frac{q}{Tx}\Big).
\end{align*}
Again, we need to choose the optimal parameter $s$ which fits best
with the Vinogradov integrals. This is provided by
the choice $s_{2}=k(k-1)/2+1$, an application of VMVT (Theorem \ref{th:VMVT})
yields
\[
\frac{J_{k-1}(x,s-1)J_{k-1}(x,s)}{x^{4s-(k-1)(k-2)-2k}}\ll 
 x^{s-1}x^{2s-k(k-1)/2}x^{-4s+(k-1)(k-2)+2k}\ll 1.
\]
Thus applying H\"older's inequality, we arrive at the assertion
\[
    \sum_{a\leq T} |S_{az}(\ua)| \ll Tx 
\Big(\frac{x^{k-1}}{qw}  +
\frac{x^{k-1}}{T}+ \frac{1}{xw}+\frac{q}{Tx}\Big)^{1/2s_{2}}.
\]
\end{proof}

We see that in the setting of Theorem~\ref{Satz3} we improved the term
$1/x$ by $1/xw$, where $w$ may be taken as large as $x^{k-1}$.
This would allow a saving of up to $x^{-1/k}$ in the estimate, namely
\[
   x^{-k/2s_{2}}=x^{-k/(k(k-1)+2)}=x^{-1/(k-1+2/k))}\ll x^{-1/k},
\]
assuming best parameter choices for $q$ and $T$ (say $x^{k}\leq
q$ and $T\geq qw$).
Like this, we come close to Conjecture~\ref{conj:c}, but 
the assumptions on $\alpha_{k}$ and $\alpha_{k-1}$ are more restrictive.

Theorem~\ref{Satz3} also suggests that there might be limitations to
Conjecture~\ref{conj:c}, such as if $\alpha_{k-1}$ is close to 
$0$ mod $1$ or has good approximation to a rational $v/w$ with small
denominator $w$. In cases like these it seems that we may not estimate the sum
in \eqref{eq:cest} much better than the way we proceed.

\section{Discrete moments of exponential sums}
\label{sec:mvexposums}

We turn now to discrete moments of general exponential sums
with smooth functions $f$. The main idea is to approximate $f$ with a
polynomial using Taylor's theorem and apply the bounds of the previous
sections.

We proceed similar as in Bordell\`es' book \cite[\S 6.6.7]{Bor},
or in Heath-Brown's recent article \cite{HB-V}.
The first result is as follows.
\begin{theorem}
\label{th:Thb}
  Let $N$ be a large positive integer, 
and let $f\in\C^{k}((0,3N))$, $k\geq 3$.
Suppose that there exists real numbers
$\lambda,A$ such that $0<\lambda\leq f^{(k)}(x)\leq A\lambda$ holds for
all $x\in (0,3N)$. Let $\rho=1/((k-2)(k-3)+2)$ and
$\mu=1+A\lambda N$.
Let $z$ be a positive integer that is considered
to be small, and let $T$ be a real number with 
$N^{-k}(zA\lambda)^{-1}\leq T\leq (zA\lambda)^{-1}$.
Then
\begin{multline}
\label{eq:Thb}
  \sum_{a\leq T} \Big|\sum_{N<m<2N} \e(azf(m))\Big| 
\ll NT(zA\lambda T)^{\rho/k+\eps} + T(zA\lambda T)^{-1/k} \\
+T\mu z^{2}(zA\lambda T)^{2/k-2} 
  + \mu z (zA\lambda T)^{1/k-1}\lambda^{-1}.
\end{multline}
\end{theorem}

We note that $\lambda$ as well as $A$ and $z$ may depend on $N$
and $T$. In the case if $A$ and $z$ is depending on $k$ only, we may hide $A$
and $z$ in the implicit constant leading to a slightly easier expression.
Additionally assuming $\mu=1$, the upper bound simplifies to
\begin{equation}
\label{eq:Thbstr}
 NT(\lambda T)^{\rho/k}+ T(\lambda T)^{-1/k}  
+ T(\lambda T)^{2/k-2} 
  + (\lambda T)^{1/k-1}\lambda^{-1}.
\end{equation}

The proof uses an adapted circle method. 
The first term in the bound \eqref{eq:Thb} respectively \eqref{eq:Thbstr}
comes from the the minor arc contribution, the
second gives a trivial contribution from a Weyl-shift, 
and the last two terms come from the major arc contribution.

\begin{proof}
Let $\Ll_{f}$ denote the left hand side of \eqref{eq:Thb}.

We start with a Weyl-shift with $1\leq H\leq N$. For this,
let $\beta_{m}=\e(azf(m))$ if $N<m<2N$, and $\beta_{m}=0$
otherwise. Then for each $h'\leq H$,
\begin{multline*}
    \sum_{N<m<2N} \e(azf(m)) = \sum_{m\in\Z} \beta_{m+h'} =
    \frac{1}{H}\sum_{h\leq H}\ \sum_{m\in\Z} \beta_{m+h}  \\
    = \frac{1}{H}\sum_{m=N-H+1}^{2N-1}\ \sum_{h\leq H} \beta_{m+h}
    = \frac{1}{H}\sum_{m=N+1}^{2N-1}\ \sum_{h\leq H} \e(azf(m+h)) +O(H).
\end{multline*}

We obtain
\begin{equation}
\label{eq:Lfglg}
  \Ll_{f} =\frac{1}{H} \sum_{a\leq T}\ \sum_{N<m<2N}\Big|\sum_{h\leq H}
  \e(azf(m+h))\Big| + O(TH).
\end{equation}

An application of Taylor's theorem provides the expansion
$f(m+h)=Q_{m}(h)+u_{m}(h)$ with
\[
Q_{m}(h)=hf'(m)+h^{2}f''(m)/2!+\dots+h^{k-1}f^{(k-1)}(m)/(k-1)!.
\]
Note that $f^{(k-1)}(m)/(k-1)!$ is the leading
coefficient of this polynomial of degree $k-1$ in $h$, and that
\[
  u_{m}(h)=f(m)+\frac{1}{(k-1)!}\int_{0}^{h} (h-v)^{k-1}f^{(k)}(m+v)dv,
\]
so that  $\e(azf(m+h))=\e(azQ_{m}(h))\e(azu_{m}(h))$.

We separate the exponential expressions containing $Q_{m}$ and $u_{m}$
by partial summation, this yields
\[
  \Ll_{f} \leq \Ss_{1}+\Ss_{2}+O(TH)
\]
with
\begin{align*}
  \Ss_{1} &\leq \frac{1}{H} \sum_{N<m<2N} \sum_{a\leq T} 
   \Big|\sum_{h\leq H}\e(azQ_{m}(h))\Big| \\
\end{align*}
and
\begin{align*}
  \Ss_{2} &\leq  \sum_{a\leq T} \frac{2\pi az}{H}\sum_{N<m<2N}
  \int_{1}^{H} \Big|\sum_{h\leq x} \e(azQ_{m}(h))\Big| \cdot |u'_{m}(x)|\dx \\
   &\ll z H^{k-2}  \sum_{N<m<2N}\int_{1}^{H}\sum_{a\leq T}a
   \Big|\sum_{h\leq x} \e(azQ_{m}(h))\Big| \dx \cdot 
   \sup_{v\in (0,H)} |f^{(k)}(m+v)| \\
 &\ll zA\lambda H^{k-2} \int_{1}^{H} \sum_{N<m<2N}
    \sum_{a\leq T}a
   \Big|\sum_{h\leq x} \e(azQ_{m}(h))\Big| \dx. 
 \end{align*}
Next, we abbreviate
\[
S_{a,m}(x):= \sum_{h\leq x} \e(azQ_{m}(h)),
\]
summarize the bounds for $\Ss_{1}$ and $\Ss_{2}$ and arrive at
\begin{multline}
\label{eq:Lfbound}
   \Ll_{f} \ll TH + \frac{1}{H} \sum_{N<m<2N} \sum_{a\leq T}
   |S_{a,m}(H)| \\ + zA\lambda H^{k-2}T
     \int_{1}^{H} \sum_{N<m<2N}\sum_{a\leq T}  |S_{a,m}(x)| \dx.
\end{multline}

For the next argument,
fix $x$ with $x\leq H\leq N$
and let $\Delta_0:=z^{-1}T^{-1}H^{1-k}$. 
Consider $m\in(N,2N)\cap\Z$ and let
\[
   \Aa_{m}:=\Big\{\alpha\in[0,1];\  \Big\|
   \frac{f^{(k-1)}(m)}{(k-1)!}-\alpha \Big\|\leq \Delta_0 \Big\}.
\]

Fix an $\alpha\in \Aa_{m}$.

We replace the leading coefficient in $Q_{m}(h)$ by $b_{k-1}\in\R$ such
that $ \frac{f^{(k-1)}(m)}{(k-1)!}-b_{k-1}\in\Z$, so that
$|b_{k-1}-\alpha|\leq \Delta_0$. Like this, we look at
\[
   f_{m}^{*}(h):=hf'(m)+\dots+h^{k-2}f^{(k-2)}(m)/(k-2)!+b_{k-1}h^{k-1}.
\]
Let $S_{a,m}^{*}(x):=\sum_{h\leq x} \e(azf_{m}^{*}(h))$, so that
$|S_{a,m}^{*}(x)|=|S_{a,m}(x)|$ and we are able to work with $S_{a,m}^{*}(x)$
instead of $S_{a,m}(x)$ in \eqref{eq:Lfbound}. Moreover, let
\[
   \tilde{f}_{m,\alpha}(h):=hf'(m)+\dots+h^{k-2}f^{(k-2)}(m)/(k-2)!+\alpha
   h^{k-1}
\]
and
\[
  \tilde{S}_{a,m}(\alpha,x):=\sum_{h\leq x} \e(az\tilde{f}_{m,\alpha}(h)).
\]
Then we have
\[
   \frac{d}{dx} (f_{m}^{*}(x)-\tilde{f}_{m,\alpha}(x)) \ll
   |b_{k-1}-\alpha|x^{k-2} \ll \Delta_0 x^{k-2},
\]
and we conclude by a partial summation, that
\begin{equation}
\label{eq:Sabsch}
   S_{a,m}^{*}(x) \ll |\tilde{S}_{a,m}(\alpha,x)| +
   az\int_{1}^{x} y^{k-2}\Delta_0 |\tilde{S}_{a,m}(\alpha,y)|\dy.
\end{equation}

Our task is reduced to prove good upper bounds for the term
\[
   \T_{x}:=\sum_{N<m<2N} \sum_{a\leq T}
   |\tilde{S}_{a,m}(\alpha,x)|
\]
with $x\leq H$. For each $m$ in the sum there is a chosen
$\alpha\in\Aa_{m}$.
We intend to apply Theorem \ref{Satz1}. We expect a good
result if we assume
$T$ to be much bigger than $zx^{k-2}$. (Note that $\deg
\tilde{S}_{a,m}(\alpha,x)=k-1$.)

For this purpose, introduce appropriate major and minor arcs.
Let
\[
   \Ma=\bigcup_{q\leq zx^{k-1}} \bigcup_{(u,q)=1}
    \Big[\frac{u}{q}-\frac{1}{qT},\frac{u}{q}+\frac{1}{qT}\Big]
\]
denote the set of major arcs, and $\ma=[0,1]\setminus \Ma$.

Now we distinguish two cases:

Say case ($\ma$) occurs if $m$ is such that 
there exists a real number $\alpha\in\Aa_{m}\cap\ma$.
We choose then such an $\alpha$ for each such $m$.
By Dirichlet's approximation theorem, there exists 
coprime integers $u$ and $q$
with $1\leq q\leq T$ such that
\[
    \Big|\alpha-\frac{u}{q}\Big|\leq \frac{1}{qT}.
\]
Since $\alpha$ is contained in $\ma$, we conclude that even $q\geq
zx^{k-1}$ holds true.

A closer look at the improvement expression in Theorem
\ref{Satz1} yields
\[
  \Big(\frac{zx^{k-2}}{q}+\frac{zx^{k-2}}{T}+\frac{1}{x}
 +\frac{q}{Tx}\Big)^{\rho} \ll x^{-\rho}
\]
with $\rho=1/((k-2)(k-3)+2)$.
Therefore by Theorem \ref{Satz1},
\[
    \sum_{a\leq T} |\tilde{S}_{a,m}(\alpha, x)|\ll Tx^{1-\rho+\eps},
\]
hence, summing up over these $m$,
\begin{equation}
  \label{eq:minor}
\T_{x,\text{($\ma$)}}:= \sideset{}{^{(\ma)}}\sum_{N<m<2N}
\ \sum_{a\leq T} |\tilde{S}_{a,m}(\alpha, x)| \ll NTx^{1-\rho+\eps}. 
\end{equation}

In the major arc case, 
$\Aa_{m}$ is contained completely in a major arc interval.
Then we conclude for $m$ with $N<m<2N$ in case ($\Ma$),
that there exist $q\leq zx^{k-1}$ and
$(u,q)=1$ such that $\|f^{(k-1)}(m)/(k-1)!-u/q\|<1/qT$.
Summing up over these $m$, we obtain
\begin{multline*}
 \T_{x,\text{($\Ma$)}}:=\sideset{}{^{(\Ma)}}\sum_{N<m<2N}
\ \sum_{a\leq T}
  |\tilde{S}_{a,m}(\alpha, x)|\\ \ll  Tx\sum_{q\leq zx^{k-1}}
  \sum_{(u,q)=1} \#\{m\in(N,2N);\ \\
  \hspace{5cm}\|f^{(k-1)}(m)/(k-1)!-u/q\|<1/qT\} \\
\ll Tx \sum_{q\leq zx^{k-1}} \varphi(q) (1+A\lambda N)(1+1/\lambda qT),
\end{multline*}
where in the last step we applied Lemma \ref{lem:smf} with
$g(x):=f^{(k-1)}(x)/(k-1)!-u/q$. 
From now on we make use of the abbreviation $\mu=1+A\lambda N$. 
This yields the bound
\[
  \T_{x,\text{($\Ma$)}} \ll Tx (\mu z^{2}x^{2k-2} + \mu zx^{k-1}\lambda^{-1} T^{-1}),
\]
hence, in the major arc case,
\begin{equation}
  \label{eq:major}
\T_{x,\text{($\Ma$)}} \ll T\mu z^{2}x^{2k-1} + \mu zx^{k}\lambda^{-1}.
\end{equation}
Then joining the estimates \eqref{eq:minor} and \eqref{eq:major} 
together yields
\begin{equation}
\label{eq:sxab}
  \T_{x}\ll \T_{x,\text{($\ma$)}} + \T_{x,\text{($\Ma$)}} \ll
  NTx^{1-\rho+\eps} + T\mu z^{2}x^{2k-1} + \mu zx^{k}\lambda^{-1}.
\end{equation}

Next, from estimate \eqref{eq:Sabsch} together with \eqref{eq:sxab} 
we obtain
\begin{multline*}
  \sum_{N<m<2N} \sum_{a\leq T} |S_{a,m}^{*}(x)| \ll
   \T_{x} + Tz\Delta_{0}\int_{1}^{x}y^{k-2}\T_{y}\dy \\
  \ll NTx^{1-\rho+\eps} +Tz\Delta_{0}\int_{1}^{x} y^{k-2} NTy^{1-\rho+\eps}\dy
  \\
  + T\mu z^{2}x^{2k-1} + Tz\Delta_{0}\int_{1}^{x} y^{k-2} T\mu
z^{2}y^{2k-1} \dy \\
+ \mu zx^{k}\lambda^{-1} +  Tz\Delta_{0}\int_{1}^{x} y^{k-2}
 \mu zy^{k}\lambda^{-1}\dy \\
 \ll NTx^{1-\rho+\eps} + T\mu z^{2}x^{2k-1}+\mu zx^{k}\lambda^{-1},
\end{multline*}
where only in the last step we used $\Delta_0=z^{-1}T^{-1}H^{1-k}$
and $x\leq H$.

Therefore, by \eqref{eq:Lfbound}, we arrive at
\begin{align}
  \Ll_{f} &\ll TH+H^{-1}\sum_{m}\sum_{a\leq T} |S_{a,m}^{*}(H)| \notag\\ &\qquad 
   +zA\lambda H^{k-2}T \int_{1}^{H} \sum_{m} \sum_{a\leq T}
   |S_{a,m}^{*}(x)|\dx \notag\\ &\ll TH+
   H^{-1}NTH^{1-\rho+\eps}+zA\lambda H^{k-2}T\int_{1}^{H}NTx^{1-\rho+\eps}\dx
   \notag\\
&+  H^{-1} T\mu z^{2}H^{2k-1} + zA\lambda H^{k-2}T\int_{1}^{H} T\mu
z^{2}x^{2k-1} \dx \notag \\ &+  H^{-1} \mu zH^{k}\lambda^{-1} 
+ zA\lambda H^{k-2}T\int_{1}^{H} \mu zx^{k-1}\lambda^{-1}\dx \notag\\
 &\ll NTH^{-\rho+\eps} + TH + T\mu z^{2}H^{2k-2} + \mu zH^{k-1}\lambda^{-1},
\label{eq:LF}
\end{align}
where we have chosen $H=[(zA\lambda T)^{-1/k}]$ in \eqref{eq:LF}.
This gives the bound
\begin{multline*}
\Ll_{f}\ll NT(zA\lambda T)^{\rho/k+\eps} + 
T(zA\lambda T)^{-1/k} \\
 +T\mu z^{2}(zA\lambda T)^{2/k-2} 
  + \mu z (zA\lambda T)^{1/k-1}\lambda^{-1}.
\end{multline*}

As necessary constraint for $T$ we get $N^{-k}\leq zA\lambda T \leq
1$, since we need $1\leq H=[(zA\lambda T)^{-1/k}]\leq N$.
\end{proof}

\textbf{Remark.} 
We have to discuss in which range for $T$ Theorem \ref{th:Thb}
provides a nontrivial upper bound for $\Ll_{F}$.

The first two terms of the bound \ref{eq:Thb}
clearly give a nontrivial upper bound $\lll TN$,
and also the third term is $\lll TN$ provided that $T\mu
z^{2}H^{2k-2}\lll TN$ which means
\begin{equation}
\label{eq:t1}
  T\ggg \mu^{k/(2k-2)}z^{1/(k-1)}(A\lambda)^{-1}N^{k/(2-2k)}.
\end{equation}
And also the fourth term is $\lll TN$ provided that
\begin{equation}
\label{eq:t2}
  T\ggg\mu^{k/(2k-1)}z^{1/(2k-1)}A^{(1-k)/(2k-1)} \lambda^{-1} N^{k/(1-2k)}.
\end{equation}
Note that this means $T\ggg \mu (\lambda N)^{-1}zH^{k-1}$, which is
stronger than just $T\ggg zH^{k-1}$ which was expected in the proof to
lead to nontrivial results.

A short calculation shows that
these lower bounds \eqref{eq:t1} and \eqref{eq:t2} 
for $T$ are admissible with the constraint $T\leq
(zA\lambda)^{-1}$ provided that $z^{2}\mu \ll N$ and $z^{2}\mu A\ll
N^{k}$. 
We conclude that then, for small $z$, there exists a range for $T$ where 
a nontrivial bound is achieved.

The lower bounds \eqref{eq:t1} and \eqref{eq:t2}  for $T$ are quite
restrictive, but realize
the advantage of Theorem \ref{Satz1} compared to Theorem \ref{Satz2}.
Using Theorem \ref{Satz2} in the proof instead will lead to the
following slightly weaker bound \eqref{eq:Thb2} since $\tau<\rho$,
but provides a larger range for $T$.

\begin{theorem}
  \label{th:Thb2}
  Let $N$ be a large positive integer and let
$f\in\C^{k}((0,3N))$, $k\geq 3$.
Suppose that there exists real numbers
$\lambda,A$ such that $0<\lambda\leq f^{(k)}(x)\leq A\lambda$ holds for
all $x\in (0,3N)$. Let $\tau=1/(k-1)(k-2)$ and
$\mu=1+A\lambda N$.
Let $z$ be a positive integer that is considered to be small 
and let $T$ be a positive real number with
$N^{-k}(zA\lambda)^{-1}\leq T\leq (zA\lambda)^{-1}$.
Then 
\begin{multline}
\label{eq:Thb2}
  \sum_{a\leq T} \Big|\sum_{N<m<2N} \e(azf(m))\Big| 
\ll  NT(zA\lambda T)^{\tau/k+\eps} +T(zA\lambda T)^{-1/k}  \\
+A\mu T  (zA\lambda T)^{-2/k}.
\end{multline}
\end{theorem}

\begin{proof}
  We proceed as before in Theorem \ref{th:Thb}, but choose now the
major arc set to be
\[
   \Ma=\bigcup_{q\leq x} \bigcup_{(u,q)=1}
    \Big[\frac{u}{q}-\frac{1}{qTzx^{k-1}},
      \frac{u}{q}+\frac{1}{qTzx^{k-1}}\Big].
\]
In the minor arc case, 
we treat $m$ with $x\leq q\leq zx^{k-1}T$
and we are in the situation
to use Theorem \ref{Satz2} instead, leading to the slightly weaker 
estimate
\[
   \T_{x,\text{($\ma$)}} \ll NTx^{1-\tau+\eps},
\]
since $\tau<\rho$, where $\tau=1/(k-1)(k-2)$.
Like this, we estimate the major arc contribution in a better way, namely
\begin{multline*}
 \T_{x,\text{($\Ma$)}}:=  \sideset{}{^{(\Ma)}}\sum_{N-H<m<2N}
  \ \sum_{a\leq T} |\tilde{S}_{a,m}(\alpha, x)|\\ \ll  Tx\sum_{q\leq x}
  \sum_{(u,q)=1} \#\{m\in(N-H,2N);\ \\
  \hspace{5cm}\|f^{(k-1)}(m)/(k-1)!-u/q\|<1/qTzx^{k-1}\} \\
\ll Tx \sum_{q\leq x} \varphi(q) (1+A\lambda N)(1+1/\lambda qTzx^{k-1}) \\
\ll Tx (\mu x^{2} + \mu x /\lambda Tzx^{k-1})
\ll \mu Tx^{3}+\mu (\lambda z)^{-1}x^{3-k},
\end{multline*}
with $\mu=1+A\lambda N$, again by using Lemma \ref{lem:smf}.
We similarly arrive at
\[
   \Ll_{f}\ll NTH^{-\tau+\eps} + TH
+ \mu TH^{2} + \mu(\lambda z)^{-1}H^{2-k},
\]
if we choose $H=[(zA\lambda T)^{-1/k}]$ again.
Since
\[
   \max\{1, (T\lambda z)^{-1}H^{-k}\} = \max\{1, A\}=A,
\]
the last two terms are $\ll A\mu TH^{2}$.
Again noting that $1\leq H\leq N$ provides the assertion
of Theorem \ref{th:Thb2}.
\end{proof}

\textbf{Remark.} Again, we give the range for $T$ where 
Theorem \ref{th:Thb2} provides a nontrivial bound for $\Ll_{f}$.

We need to inspect the
third term in this bound, it is $\lll TN$ 
provided that $A\mu TH^{2}\lll TN$ which means 
\begin{equation}
\label{eq:tt1}
T\ggg \mu^{k/2}A^{k/2-1}(z\lambda)^{-1}N^{-k/2}.
\end{equation}

A short calculations shows that this lower bound \eqref{eq:tt1}
for $T$ is admissible with the constraint $T\leq (zA\lambda)^{-1}$
provided that $\mu A\ll N$. 

Compared to \eqref{eq:t1} and \eqref{eq:t2}, the range for $T$
due to \eqref{eq:tt1} will be much bigger in most cases.

\medskip
We compare our theorems with the direct application of the 
following recent result of Heath-Brown in \cite[Thm.~1]{HB-V}. 
\begin{theorem}[Heath-Brown]
  \label{th:HB-V}
  Let $k\geq 3$, let $f:[0,N]\to\R$ denote a function in
  $C^{k}((0,N))$, and suppose that $0<\lambda\leq 
  f^{(k)}(x)\leq A\lambda$ for all $x\in(0,N)$. Then
  \[
     \sum_{n\leq N} \e(f(n)) \ll_{A,k,\eps} N^{1+\eps} (\lambda^{1/k(k-1)} +
     N^{-1/k(k-1)} + N^{-2/k(k-1)}\lambda^{-2/k^{2}(k-1)}).
  \]
\end{theorem}

In principle, 
$\Ll_{f}$ can be estimated by using Theorem~\ref{th:HB-V}, 
but one needs then the dependence of the implicit constant 
on $A$ explicitly since the term
$az$ occurs in the argument of the complex exponential
function, so that $A$ in Theorem~\ref{th:HB-V} 
contains this factor $az\leq zT$. 

Writing down the dependence on $A$ from the proof
in \cite[Thm.~1]{HB-V} explicitly, 
we will have a factor $A^{4}$ occurring in the quantity $\Nn$ there.
The resulting bound for $\Ll_{f}$ will then
contain the factor $A^{4/2s}=A^{4/k(k-1)}$. Thus the
main term from this method will provide the extra factor
$T^{4/k(k-1)}$ which is much larger than the factors $T^{\rho/k}$ or
$T^{\tau/k}$ from our Theorems here.

So compared to this, Theorems \ref{th:Thb} and \ref{th:Thb2} give
sharper estimates for long Weyl sum averages. When no or short 
averages are considered, Heath-Brown's bound is sharper.

Note that the potential improvements depend also on the type
of functions considered. For example, if $f(n)=t\log n$ and $k$ is
large with $t=N^{k/2}$, then the minor arc contribution from 
Theorem~\ref{th:Thb} is around $tN^{1-1/k(k-1)(k-2)}$, 
whereas Heath-Brown's bound is around $tN^{1-1/k^{2}}$. It may be that 
Theorem~\ref{th:Thb}, as it stands,
can \emph{not} be improved if uniformity for \emph{all} functions is
desired, but further improvements 
may be feasible if a suitable type of function is assumed.

A careful combination of the methods from
Theorem~\ref{Satz1} 
and Theorem~\ref{th:Thb} together with the idea from the proof of
Theorem~\ref{Satz3} could lead to further improvements, which is 
not discussed in this article.
%

\section{First application: Integer points close to smooth curves}
\label{sec:appl1}

In this application, we will use Theorem \ref{th:Thb2}.
We introduce the following quantity.

\begin{definition}\label{defR}
Let $N$ be a large positive integer, let
$f\in\C^{k}((0,3N))$, $k\geq 3$ and $0<\delta<1/4$.
Define
\[
\Rr(f,N,\delta):=\#\{n\in[N,2N]\cap\Z;\ \|f(n)\|<\delta \}.
\]
\end{definition}
Like this, we count lattice points in $\Z^{2}$ close to the
graph of $f$.
Bordell\`es gives in \cite[Ch.~5]{Bor} an overview of
several known nontrivial bounds for this quantity and their
applications.
In Lemma \ref{lem:smf} we gave already a bound for $\Rr(f,N,\delta)$ 
in the case $k=1$, it is also known as the first derivative test
for $\Rr(f,N,\delta)$. 

In what follows, we use a property of the set which is
counted by $\Rr(f,N,\delta)$. It is proved in \cite[Thm.~5.11]{Bor}.
\begin{lemma}
  \label{lem:space}
Suppose that there exists real numbers
$\lambda, A>0$
such that $\lambda\leq f^{(k)}(x)\leq A\lambda\leq 1/4$ holds for
all $x\in(0,3N)$.
  In the set $\{n\in[N,2N]\cap\Z;\ \|f(n)\|<\delta \}$ there exists a
$H'$-spaced subset $\Rr$ such that $\Rr(f,N,\delta)\leq (k+1)(1+\#\Rr)$,
where $H'=(A\lambda)^{-2/k(k+1)}$.
\end{lemma}

A set is called $H'$-spaced 
if any two elements differ by more than $H'>0$.

Theorem \ref{th:Thb2} of Section \ref{sec:mvexposums} 
allows us to prove a strong bound for $\Rr(f,N,\delta)$
which is the following. 

\begin{theorem}
\label{th:thR}
Suppose that there exists real numbers
$\lambda, A>0$
such that $\lambda\leq f^{(k)}(x)\leq A\lambda\leq c_{0}$ holds for
all $x\in(0,3N)$ and some small constant $c_{0}<1/4$,
and assume $(A\lambda)^{-2/k(k+1)}\leq N$.
Let $\lambda_{1}>0$ be such that $|f'(x)|\leq \lambda_{1}$ holds for all
$x\in(0,3N)$.  
Assume that $\delta>0$ is such that
\begin{equation}
\label{eq:condthR}
  A\lambda\ll \delta+ \lambda_{1}\ll (A\lambda)^{1-2(k-1)/k(k+1)},
\end{equation}
then we have the bound
\begin{equation}
\label{eq:Rb}
    \Rr(f,N,\delta)\ll 
1+ A(1+A\lambda N)((\delta+\lambda_{1})/A\lambda)^{1/(k-1)}.
\end{equation}
\end{theorem}

\begin{proof}
We begin the proof as indicated in \cite[Ex.\ 6.7.4]{Bor}.
Let $m\in\Z$ with $N\leq m\leq 2N$ and $\|f(m)\|<\delta$,
then
\begin{multline}
\label{eq:Tqu}
 \Big|\sum_{a=0}^{T-1}\e(af(m))\Big|^{2}=\sum_{a_{1},a_{2}}\e((a_{1}-a_{2})f(m))
 \\ = \sum_{a_{1},a_{2}}\Re(\e((a_{1}-a_{2})f(m))) \geq T^{2}/2,
\end{multline}
since we have $\Re(\e(af(m)))\geq \sqrt{2}/2$ for all $a\in\Z$ with $|a|<T$,
provided that $1\leq T\leq [1/8\delta]+1$.

From Lemma \ref{lem:space} we know that there is a $H'$-spaced subset $\Rr$ of 
$\{m;\ N\leq m\leq 2N, \|f(m)\|<\delta\}$ with $\Rr(f,N,\delta)\leq
(k+1)(1+\#\Rr)$, where $H'= (A\lambda)^{-2/k(k+1)}$.

Hence 
\[
    \Rr(f,N,\delta)\ll 1+\sum_{m\in \Rr} 1\leq 1+
\frac{2}{T^{2}} \sum_{m\in\Rr} \Big|\sum_{a=0}^{T-1}\e(af(m))\Big|^{2}.
\]
Now opening the square and separating the summand for $a=0$ shows
\begin{multline}
\label{eq:amdef}
\Big|\sum_{a=0}^{T-1}\e(af(m))\Big|^{2} = \sum_{|a|\leq T-1}
(T-|a|)\e(af(m))
\\ = T + \sum_{0<|a|\leq T-1}(T-|a|)\e(af(m))=:T+a_{m} 
\end{multline}
say. Clearly $a_{m}>0$ for large $T$ and for $m\in\Rr$ due to \eqref{eq:Tqu}.
We proceed with
\begin{equation}
\label{eq:Rexposum}
   \Rr(f,N,\delta)\ll 1+ \Rr(f,N,\delta)T^{-1} +
   T^{-2}\sum_{m\in\Rr}a_{m} \ll 1+T^{-2}\sum_{m\in\Rr}a_{m},
\end{equation}
by an application of Lemma \ref{lem:sri} if we take $T\gg 1$ (in fact
$T\geq 4(k+1)$ suffices).

Let $m\in\Rr$,
then there exists an integer $n$ with $|f(m)-n|<\delta$. Thus,
by an application of the mean value theorem,
$|f(m-h)-f(m)|=|f'(t)|h\leq \lambda_{1}H$ 
for some $t\in (-h,0)$.
We conclude $|f(m-h)-n|\leq|f(m-h)-f(m)|+|f(m)-n|\leq\lambda_{1}H+\delta$, so
$\|f(m-h)\|\leq \lambda_{1}H+\delta$ for all $h\leq H$.

This argument shows that for $m\in\Rr$ 
we have $a_{m-h}\geq T^{2}/2-T\gg T^{2}$ by \eqref{eq:Tqu}
assuming 
\begin{equation}
\label{eq:Tbound}
T\leq \frac{1}{8(\lambda_{1}H+\delta)}.
\end{equation}
Since $T^{2}\geq a_{m}$ we conclude that
\[
a_{m}\ll \frac{1}{H}\sum_{h\leq H}a_{m-h}.
\]
This implies
\[ 
\sum_{m\in\Rr}a_{m}\ll \sum_{m\in\Rr}\frac{1}{H}\sum_{h\leq
  H}a_{m-h}. 
\]

Now we have
\[\sum_{m\in\Rr}\frac{1}{H}\sum_{h\leq
  H}a_{m-h}  
= \frac{1}{H}\sum_{\substack{n\in\Z\\ \exists h\leq H:
   n+h\in\Rr}} a_{n}, 
\]
since for each $m\in\Rr$, the summands $a_{m-1},\dots,a_{m-h}$ occur
in the sum on the right hand side. Assuming $H\leq H'$, the elements
of $\Rr$ are $H$-spaced, so each such summand occurs exactly once on
both sides of this equality.

This improves upon the Weyl step in the proofs above, namely
\begin{multline*}
  \sum_{m\in\Rr} a_{m} \ll 
    H^{-1} \sum_{\substack{n\in\Z\\\exists h'\leq H: n+h'\in\Rr}}
a_{n} 
  = H^{-2}\sum_{h\leq H} \sum_{\substack{n\in\Z\\\exists h'\leq H:
      n+h+h'\in\Rr}} a_{n+h} \\
  = H^{-2} \sum_{\substack{n\in\Z\\\exists h_{0}\in[1,2H]:
      n+h_{0}\in\Rr}}\ \sum_{\substack{h\leq H \\
      h\in[h_{0}-H,h_{0}-1] }} a_{n+h}\\
  = H^{-2} \sum_{\substack{n\in\Z\cap[N-2H,2N]\\\exists h_{0}\in[1,2H]:
      n+h_{0}\in\Rr}}\ \sum_{\substack{h\leq H \\
      h\in[h_{0}-H,h_{0}-1] }} a_{n+h}. 
\end{multline*}

Note that the number of $n\in\Z$ in this sum is $\ll
\min\{H\#\Rr,N\}=:R$, which gives an improvement for a sparse set
$\Rr$ since then $H\#\Rr\ll N$.
Further, $h$ lies in the intersection $\Hh:=[1,H]\cap[h_{0}-H, h_{0}-1]$
 which is an interval
of length at most $H$. The extra factor
$H^{-1}$ due to the improved Weyl step will be an advantage in the
minor arc analysis.

With above definition \eqref{eq:amdef} of $a_{m}$,
we derive
\begin{equation}
\label{eq:eqsuma}
\sum_{m\in\Rr} a_{m}\ll TH^{-2} \sum_{a\leq T} \sum_{\substack{m\in[N-2H,2N]\\
    \exists h_{0}\leq 2H:m+h_{0}\in\Rr}}
\Big|\sum_{h\in\Hh} \e(af(m+h))\Big|. 
\end{equation}

The right hand side in \eqref{eq:eqsuma} can now
be handled like $\Ll_{f}$ in the proofs of the Theorems 
\ref{th:Thb} and \ref{th:Thb2} above.
For this, let us denote in analogy to $\Ll_{f}$, 
\[
\Ll_{f}':= H^{-1} \sum_{a\leq T}  \sum_{\substack{m\in[N-2H,2N]\\
    \exists h_{0}\leq 2H:m+h_{0}\in\Rr}} \Big|\sum_{h\in\Hh} \e(af(m+h))\Big|.
\]
Thus
\begin{equation}
\label{eq:rr}
\Rr(f,N,\delta)\ll 1+ T^{-1}H^{-1}\Ll_{f}'
\end{equation}
by \eqref{eq:Rexposum}
and \eqref{eq:eqsuma}. Note that $\Ll_{f}$ differs from $\Ll_{f}'$
only by the summation over $m$ and that $h$ runs through an interval
$\Hh$ of length 
at most $H$, which boundary points depend on $m$.

The estimation of $\Ll_{f}'$ follows now that of $\Ll_{f}$.
The only small change in above proof of Theorem
\ref{th:Thb2} lies in the minor arc estimate 
\eqref{eq:minor}, where we are able to replace $N$ by $R$ so that
\[
   \T'_{x,(\ma)} \ll RTx^{1-\tau+\eps}.
\]
To verify this, note that 
\[
   \T_{x}':=\sum_{\substack{N-2H<m<2N\\ \exists h_{0}\leq 2H:m+h_{0}\in\Rr}} \sum_{a\leq T}
   |\tilde{S}_{a,m}'(\alpha,x)|
\]
involves a restriction of $h$ in the sum 
\[
  \tilde{S}_{a,m}'(\alpha,x)=\sum_{h\in\Hh,h\leq x} \e(a\tilde{f}_{m,\alpha}(h))
\]
to $h\in\Hh$.
Still Theorem \ref{Satz2} can be applied since the condition $h\in \Hh$ 
restricts the summation down to a set which is an interval.
Since this interval has an upper bound of at most $x$, 
this provides the stated bound for $\T'_{x,(\ma)}$.

Next, for appropriate $H$, choose $T=[\frac{1}{H^{k}A\lambda}]$. 
Especially, if $H^{k}A\lambda -8\lambda_{1}H\geq 8\delta$ 
then \eqref{eq:Tbound} is true.
We obtain in the same way as in
Theorem \ref{th:Thb2} that
\[
   \Ll_{f}'\ll RTH^{-\tau+\eps} + A\mu TH^{2}.
\]

We put this bound inside \eqref{eq:rr} above and obtain
\begin{multline*}
  \Rr(f,N,\delta) \ll 1+ T^{-1}H^{-1}(RTH^{-\tau+\eps} + A\mu TH^{2}) 
\\ \ll 
  1+\Rr(f,N,\delta) H^{-\tau+\eps} + A\mu H.
\end{multline*}

By an application of Lemma \ref{lem:sri} we leave out the term on
the right hand side containing $\Rr(f,N,\delta)$ assuming that $H$ is
large enough in terms of the implicit constant.
This works since $H^{-\tau+\eps}$ gets
arbitrary small if $H$ increases.

We arrive at the bound 
\[
\Rr(f,N,\delta)\ll 1+ A(1+A\lambda N)H. 
\]

Now we collect all the assumptions made on $H$.
Due to Lemma \ref{lem:space} we need $H\leq (A\lambda)^{-2/k(k+1)}
\leq N$.

Moreover, due to \eqref{eq:Tbound} we need
$H^{k}A\lambda-8\lambda_{1}H\geq 8\delta$,
that is if
\begin{equation}
\label{eq:deltab}
  8\delta \leq H(H^{k-1}A\lambda-8\lambda_{1}),
\end{equation}
for which necessarily $H^{k-1}A\lambda>8\lambda_{1}$ has to be true.
Let
$H=((8\delta+8\lambda_{1})/A\lambda)^{1/(k-1)}$,
so that \eqref{eq:deltab} holds true.
Now if
\[
   ((\delta+\lambda_{1})/A\lambda)^{1/(k-1)} \ll (A\lambda)^{-2/k(k+1)},
\]
that is, if
\begin{equation}
\label{eq:cond}
  \delta+ \lambda_{1}\ll (A\lambda)^{1-2(k-1)/k(k+1)}
\end{equation}
holds true,
we conclude that  $H$ is appropriate
for the asserted bound in the theorem to be valid
for all small $\delta$ such that \eqref{eq:cond} holds.
We further need to assume that $H$ is bigger
than some constant which makes above step with Lemma \ref{lem:sri}
work, so we shall assume also $A\lambda \ll \delta+\lambda_{1}$.

This yields the assertion.
\end{proof}

\medskip
We shall compare the bound in Theorem \ref{th:thR} 
with the well-known theorem of Huxley and Sargos from 
\cite[Thm.~1]{HuxSar}, it states the following bound for $\Rr(f,N,\delta)$.
The given version here is explicit in $A$ and has been taken from
\cite[Thm.~5.12]{Bor} where a proof is provided. The known proofs
are geometric and do not depend on any exponential sum technique.

\begin{theorem}[Huxley and Sargos, explicit in $A$]
\label{th:HS}
Let $k\geq 3$ be an integer and $f\in C^{k}([N,2N])$ be such that there
exist $\lambda,A>0$ with $\lambda \leq |f^{(k)}(x)|\leq A\lambda$
for all $x\in[N,2N]$. Let $0<\delta<1/4$. Then
\begin{equation*}
   \Rr(f,N,\delta)\ll 
     N(A\lambda)^{2/k(k+1)} +
     N(A\delta)^{2/k(k-1)} + (\delta/\lambda)^{1/k}+1. 
\end{equation*}
\end{theorem}

In Theorem \ref{th:HS}, the first term
$N(A\lambda)^{2/k(k+1)}$ dominates if
$\delta\leq\lambda^{1-2/(k+1)}$
 $A^{-2/(k+1)}$
and $\delta\ll N^{k}A^{2/(k+1)}\lambda^{1+2/(k+1)}$,
especially when examining very small $\delta$.
The main term $N(A\lambda)^{2/k(k+1)}$ 
is commonly called the smoothness term 
and being regarded as very difficult to improve,
compare also \cite[p.275]{Bor}.

A comparison of the term $NA^{2}\lambda((\delta+\lambda_{1})/A\lambda)^{1/(k-1)}$
in the bound of Theorem \ref{th:thR} with the smoothness term
shows that under the stated assumptions,
Theorem \ref{th:thR} gives a sharper bound.
Also note that if we let $\delta\to 0$ 
in Theorem \ref{th:thR}, we obtain the improved bound
$\Rr(f,N,0)\ll 1+ A(1+A\lambda N)(\lambda_{1}/A\lambda)^{1/(k-1)}$,
compared to Theorem \ref{th:thR} which yields
the smoothness term as upper bound for $\Rr(f,N,0)$.

The bound of Theorem \ref{th:thR} is quite satisfying,
the only obstruction lies in the restrictive assumption
\begin{equation}
\label{eq:obstr}
\lambda_{1}\ll (A\lambda)^{1-2(k-1)/k(k+1)}
\end{equation}
needed for $\lambda_{1}$ and $A\lambda$.
This shows that Theorem \ref{th:thR} 
is in fact of limited use in applications since this condition
is true only for certain appropriate functions. Unfortunately,
functions like $f(x)=B/x^{r}$ or $f(x)=(B/x)^{1/r}$ for
integers $r\geq 2$, which are occurring in interesting applications,
are not of this kind when $A$ is taken as constant.
And if $A$ is so big such that \eqref{eq:obstr} holds, the bound of Theorem
\ref{th:thR} can be quite weak.

In this context, we state the following theorem of Gorny \cite{Gor}.
\begin{theorem}[Gorny]
  \label{th:gor}
    Let $k\geq 2$, $f\in C^{k}([1,1+3N])$. Let $M,A,\lambda\in\R$ such that
  $|f(x)|\leq M$ and $|f^{(k)}(x)|\leq A\lambda$ holds for all $x\in [1,1+3N]$.
  Then for all $x\in [1,1+3N]$,
\[
    |f'(x)| \ll MN^{-1} + M^{1-1/k}(A\lambda)^{1/k},
\]  
with implicit constant that depends on $k$ only.
\end{theorem}

By this theorem, we conclude that \eqref{eq:obstr} is true for all
sufficiently large $N$ provided that $M\leq (A\lambda)^{1-2/(k+1)}$.
Therefore, functions on $[1,1+3N]$ that are much smaller in absolute value 
compared to the maximum of the absolute value of the
$k$-th derivative are admissible for Theorem \ref{th:thR}.
This gives a nice criterion for Theorem \ref{th:thR} to hold, but it
seems to be hard to find easy examples.

\textbf{Remark.}
Applying Theorem \ref{th:Thb} instead of Theorem \ref{th:Thb2} in
the proof of Theorem \ref{th:thR} would 
also lead to a vanishing of the minor arc contribution in the bound,
but the resulting bound for $\Rr(f,N,\delta)$ 
would be weaker due to the bigger major arc contribution.
Instead, the presented vanishing trick may be used with even
bigger minor arc contributions probably leading to further refinements.

\section{Second application:
The polynomial large sieve inequality 
(LSI) in  the one-dimensional case}
\label{sec:polyLSI}
In this section we present an application of Theorem \ref{Satz1},
namely to the polynomial large sieve inequality. This generalization of the
classical large sieve inequality to sparse moduli sets, usually given as values
of some fixed polynomial, has been studied intensely in past research and has
already influenced some other areas of Number Theory, especially the
version with $k=2$ from \cite{BaiZhao}.
To name a few such topics, it has been found useful to find
variants of Bombieri--Vinogradov's theorem \cite{Bak,KHBVvar},
new results on 
primes of polynomial shape \cite{FooZhao} or primes in APs to spaced
moduli \cite{Bak},
divisibility questions with Fermat quotients \cite{BFKS}, 
and mean value estimates for character sums with
applications \cite{BaiYou,BaiZhao2,BloGoLa, LamYou}. Furthermore, the
multidimensional polynomial LSI can be used 
for sieving with high powers like seen in \cite{KH2} and reaches questions
related to the abc-conjecture.

We start by giving the setting and basic assumptions in the polynomial
LSI in the one-dimensional case.

\textbf{Setting.} Let $P\in\R[x]$ be a fixed monic polynomial of
degree $k\geq 2$ with $P(0)=0$.
Assume that $P$ has only positive values in $[Q,2Q]$ for each real $Q\geq 1$
and let $M_{Q}:=\max \{P(q);\ q\in [Q,2Q]\}$ be the maximal value for
integers $q\in[Q,2Q]$. Clearly $M_{Q}\ll Q^{k}$,
assume also that $P(q)\gg Q^{k}$ holds true for all
integers $q\in[Q,2Q]$ and some implicit constant that may depend only on $k$.
Let $N,M$ be integers and $(v_{n})_{n\in\N}$ be a sequence of complex
numbers. 

In the theory of the large polynomial LSI, see
\cite{BaiZhao,KH,Zhao}, 
we aim to give upper bounds for the quantity 
\[
  \Sigma_{P}:=\sum_{q\leq Q} \sum_{\substack{1\leq a\leq
      P(q)\\(a,P(q))=1}} \Big| \sum_{M<n\leq M+N} 
    v_{n}\e\Big(\frac{an}{P(q)}\Big) \Big|^{2}.
\]
When we put the current form for Weyl's inequality,
Theorem~\ref{th:weyloriginal}, in the machinery of  \cite{KH,KH2},
the bound $\Sigma_{P}\ll Q^{\eps} \|v\|^{2}
(Q^{k+1}+A_{k}(Q,N))$ is easily derived with
\[
   A_{k}(Q,N):= NQ^{1-1/k(k-1)}+N^{1-1/k(k-1)}Q^{1+1/(k-1)}.
\] 

The interesting range for $N$ is $Q^{k}\ll N\ll Q^{2k}$, since outside
that, it is already known by an application of the standard large
sieve inequality that the sharp bound 
$\Sigma_{P}\ll Q^{\eps} \|v\|^{2} (Q^{k+1}+N)$ holds true. So we
assume without loss of generality that $N$ lies in this range.

This result already offers an improvement compared to \cite{BaiZhao2}
in the range $Q^{k}\ll N \ll Q^{2k-2+2/k(k-1)}$ when $k\geq 3$. We
will now improve on that.

\subsection{The connection of the Polynomial LSI with Weyl sums}
\label{sub:integral}

From \cite[Lemma~1]{KH}, we know that
  \begin{equation}\label{eq:Fint}
    \Sigma_{P} \ll Q^{\eps}\|v\|^{2}\bigg(\sum_{Q<q\leq 2Q} P(q) +
    \max_{Q<r\leq 2Q} \max_{\substack{1\leq
        b<P(r)\\\gcd(b,P(r))=1}}  
      \int_{1/N}^{1/2} \#\mathcal{F}_{b,P(r)}(x) \frac{dx}{x^{2}}\bigg).
  \end{equation}

The first term in large brackets 
can be estimated as $\ll QM_{Q}$ and is admissible.
It is not necessary to repeat the definition of $\#\mathcal{F}_{b,P(r)}(x)$
since we will just make use of the upper bound
\[
   \#\mathcal{F}_{b,P(r)}(x) \ll B^{-1} Q + B^{-1}\sum_{1\leq a\leq B}|S_{a}|
\]
with $B^{-1}=2M_{Q}x$ and
\[
    S_{a}:=\sum_{q\leq 2Q} \e\Big(\frac{ab}{P(r)}P(q)\Big),
\]
which has been shown in the deduction of \cite[(8)]{KH}.

To consider the integral expression in \eqref{eq:Fint},
fix a pair $b,r$ with $r\in[Q,2Q]$, $1\leq b<P(r)$ and
  $\gcd(b,P(r))=1$.  
We substitute $B^{-1}=2M_{Q}x$ 
and estimate as follows.
 \begin{multline*}
   \begin{aligned} 
\int_{1/N}^{1/2} &\#\mathcal{F}_{b,P(r)}(x) \frac{dx}{x^{2}}\\
   &\ll   \int_{1/4M_{Q}}^{N/2M_{Q}} \bigg(B^{-1}Q +
B^{-1}\sum_{a\leq B} |S_{a}|\bigg) M_{Q} \,dB \\     
&\ll QM_{Q}\log N + M_{Q} \int_{1/4M_{Q}}^{N/2M_{Q}} 
B^{-1}\sum_{a\leq B} |S_{a}| \,dB \\ 
&\ll QM_{Q}\log N + M_{Q} \sum_{a\leq N/2M_{Q}} |S_{a}|
\int_{a}^{N/2M_{Q}} B^{-1} \,dB \\
&\ll QM_{Q}\log N + M_{Q} \log N \sum_{a\leq N/2M_{Q}}|S_{a}|.
  \end{aligned}
  \end{multline*}

Here the last sum is a discrete moment of a Weyl sum 
with the polynomial $bP(x)/P(r)$ and leading term $b/P(r)$ since $P$
is monic.
We are able to apply Theorem \ref{Satz1} directly with $P(r)$ as approximating
denominator (when $k\geq 3$). By this, we have 
\[
   \sum_{a\leq N/2M_{Q}}|S_{a}| \ll \frac{N}{M_{Q}}Q^{1+\eps} 
     \Big(\frac{Q^{k-1}}{P(r)} + \frac{Q^{k-1}}{N/M_{Q}} + \frac{1}{Q}
     + \frac{P(r)}{QN/M_{Q}}\Big)^{1/2s_{0}}
\]
with
\[s_{0}=(k-1)(k-2)/2+1, \text{ so that }2s_{0}=k(k-1)-2k+4.
\]
Let $\omega:=1/2s_{0}$. 

In the big bracket expression, the last summand $P(r)M_{Q}/NQ$
dominates since $P(r)\gg Q^{k}$ and $1/Q\ll P(r)M_{Q}/NQ$ for $N\ll
Q^{2k}$.  

So we continue with
\begin{align*}
  \int_{1/N}^{1/2} &\#\mathcal{F}_{b,P(r)}(x) \frac{dx}{x^{2}}\\
&\ll  QM_{Q}\log Q + N^{1-\omega}M_{Q}^{\omega}
 Q^{1-\omega+\varepsilon} P(r)^{\omega} \\
&\ll QM_{Q}\log Q +  N^{1-\omega} M_{Q}^{2\omega}
  Q^{1-\omega+\varepsilon} \\
&\ll Q^{k+1}\log Q +  N^{1-\omega} Q^{1+2k\omega-\omega+\varepsilon},
\end{align*}
where we used $M_{Q}\ll Q^{k}$ in the last step.

Compared to the dominating term $NQ^{1-1/k(k-1)}$ in the former bound
$A_{k}(Q,N)$,
we get an advantage if $N^{1-\omega}Q^{1+(2k-1)\omega}\leq
NQ^{1-1/k(k-1)}$, 
which is the case if
$N\geq Q^{2k-2/(k-1)+4/k(k-1)}$, so when $N$ is close to $Q^{2k}$,
but still in the interesting range $Q^{k}\ll N\ll Q^{2k}$.
We have therefore shown the following new improved bound
for the polynomial LSI.
\begin{theorem}
\label{th:newPLSI}
In the setting of Section~\ref{sec:polyLSI} when $k\geq 3$,
\begin{equation}
\label{eq:newPLSI}
  \Sigma_{P}\ll Q^{\varepsilon} \|v\|^{2} (Q^{k+1} +
  \min\{A_{k}(Q,N), N^{1-\omega} Q^{1+(2k-1)\omega}\})
\end{equation}
with $\omega=1/((k-1)(k-2)+2)$.
\end{theorem}
Theorem~\ref{th:newPLSI} offers an improvement compared to known previous
results when $k\geq 4$. This is since for $k=3$, the additional bound from
\cite{BaiZhao2} is still stronger in this case.

It is interesting what we would obtain having Conjecture~\ref{conj:c}.
In this case, we would be able to 
gain a factor $Q^{(1-k)\omega}$.
Then, we would arrive at the following result.

\begin{conjecture}
\label{conj1} In the setting of Section~\ref{sec:polyLSI} when $k\geq 3$,
\[
  \Sigma_{P}\ll Q^{\varepsilon} \|v\|^{2} (Q^{k+1} + \min\{A_{k}(Q,N),
  N^{1-\omega} Q^{1+k\omega}\})
\]
with $\omega=1/((k-1)(k-2)+2)$.
\end{conjecture}
Note that if we could take even $1/k(k-1)$ at the place of $\omega$,
the expression $N^{1-1/k(k-1)} Q^{1+1/(k-1)}$ coincides with the
second summand in $A_{k}(Q,N)$.

Still, these conjectural bounds are far from Zhao's conjecture
in \cite{Zhao} stating 
\[
  \Sigma_{P}\ll Q^{\varepsilon} \|v\|^{2} (Q^{k+1} + N).
\]
Conjecture \ref{conj1} might be rather within reach 
of further refinements of the methods
presented in this article.

We remark that an attempt to use Theorem~\ref{Satz3} will not give
more if more on the coefficient of $q^{k-1}$ in $P(q)$, say
$\alpha_{k-1}$, is known. This is since the coefficient of $q^{k-1}$
in the polynomial $bP(q)/P(r)$ is then $\alpha_{k-1}b/P(r)$. By
Theorem~\ref{Satz3}, one needs to look then at the rational approximations
to $\alpha_{k-1}b$. Since $b$ is supposed to be \emph{any} coprime
residue mod $P(r)$, there will always be one with small denominator
which offers no advantage in Theorem~\ref{Satz3}.


\section{Acknowledgements}
The author thanks the organizers of the Workshop
on Efficient Congruencing in March 2017
at the Fields institute in Toronto during the Thematic Program 
on Unlikely Intersections, Heights, and Efficient Congruencing
for an inspiring stay. The author also 
thanks the referee for many useful suggestions and comments on the
manuscript which led to an improvement of some of the 
material.


\end{document}